\documentclass[a4paper,12pt]{amsart}
\usepackage{latexsym}
\usepackage{amsmath,amssymb}
\usepackage{amsmath}

\usepackage[all]{xy}
\usepackage{enumerate}
\usepackage[usenames]{color}
\usepackage{xcolor}

\newtheorem{theo}{Theorem}[section]
\newtheorem{coll}[theo]{Corollary}
\newtheorem{lemm}[theo]{Lemma}
\newtheorem{prop}[theo]{Proposition}
\newtheorem{defn}[theo]{Definition}
\newtheorem{ex}[theo]{Example}
\newtheorem{rem}[theo]{Remark}

\newcommand{\Hom}{{\rm Hom}}
\newcommand{\End}{{\rm End}}

\newcommand{\A}{\mathcal A}

\newcommand{\C}{\mathcal C}

\begin{document}
\sloppy

\title[CS-Rickart and dual CS-Rickart objects]{CS-Rickart and dual CS-Rickart objects \\ in abelian categories}

\author[S. Crivei]{Septimiu Crivei}

\address{Faculty of Mathematics and Computer Science, Babe\c s-Bolyai University, Str. M. Kog\u alniceanu 1,
400084 Cluj-Napoca, Romania} \email{crivei@math.ubbcluj.ro}

\author[S.M. Radu]{Simona Maria Radu}

\address{Faculty of Mathematics and Computer Science, Babe\c s-Bolyai University, Str. M. Kog\u alniceanu 1,
400084 Cluj-Napoca, Romania} \email{simonamariar@math.ubbcluj.ro}



\begin{abstract} We introduce (dual) relative CS-Rickart objects in abelian categories,
as common generalizations of (dual) relative Rickart objects and extending (lifting) objects. 
We study direct summands and (co)products of (dual) relative CS-Rickart objects as well as 
classes all of whose objects are (dual) self-CS-Rickart. Applications are given to Grothendieck categories
and, in particular, to module and comodule categories.
\end{abstract}

\date{September 1, 2021}

\maketitle

\section{Introduction}

A basic module-theoretic result states that a module $M$ is semisimple if and only if every submodule of $M$ is a direct summand. One may obtain various generalizations of semisimplicity by considering only some submodules of a given module to be direct summands. For instance, if $M$ is the right $R$-module $R$, then $R$ is (von Neumann) regular if and only if every finitely generated submodule of $M$ (i.e., right ideal of $R$) is a direct summand. Given a module $M$, one may also define an extension of semisimplicity by using certain submodules related to the endomorphisms of $M$.  As such, a module $M$ is called Rickart if the kernel of every endomorphism of $M$ is a direct summand of $M$ \cite{LRR10}. Using a different approach, one may consider another generalization of semisimplicity, namely: a module $M$ is called extending (or CS-module) if every submodule of $M$ is essential in a direct summand \cite{DHSW}. A natural question is what happens when one restricts the definition of an extending module $M$ only to some submodules related to the endomorphisms of $M$ in the same style in which one obtains the concept of Rickart module from that of semisimple module? This is the way to obtain the so-called CS-Rickart modules, which are defined as modules $M$ such that the kernel of every endomorphism of $M$ is a direct summand of $M$ \cite{AN1}. We point out that  CS-Rickart modules may be viewed as playing the role of extending modules in the world of Rickart modules. The properties of CS-Rickart modules are sometimes similar to those of Rickart modules, but one needs different techniques to obtain them, in the same way as one uses different approaches to study extending modules in contrast to semisimple modules.

Rickart and dual Rickart objects in abelian categories have been introduced and studied by 
Crivei, K\"or and Olteanu \cite{CK,CO}. On one hand, they generalize regular objects in abelian categories in 
the sense of D\u asc\u alescu, N\u ast\u asescu, Tudorache and D\u au\c s \cite{DNTD}. 
Thus, an object is regular if and only if it is both Rickart and dual Rickart.
The main interest in their study stems from the work of von Neumann \cite{vN} 
on regular rings and Fieldhouse \cite{Field} and Zelmanowitz \cite{Zel} on certain concepts of regular modules. 
The main idea of the approach from \cite{CK} was to split the study of regular objects into two directions,
one of Rickart objects and the other one of dual Rickart objects. Since they are dual concepts, 
it is enough to study one of them followed by the use of the duality principle in abelian categories.

On the other hand, Rickart and dual Rickart objects generalize to abelian categories  
Rickart and dual Rickart modules in the sense of Lee, Rizvi and Roman \cite{LRR10,LRR11}, 
and in particular, Baer and dual Baer modules studied by Rizvi and Roman \cite{RR04,RR09} 
and Keskin T\"ut\"unc\"u and Tribak \cite{KT} respectively. The origin of (dual) Baer modules and (dual) Rickart modules  
can be found in the work of Kaplansky \cite{K} on Baer rings and Maeda \cite{Maeda} on Rickart rings respectively.
Examples of Baer rings include von Neumann regular right self-injective rings, von Neumann algebras, 
endomorphism rings of semisimple modules, while examples of Rickart rings include Baer rings, von Neumann regular rings, 
right hereditary rings, endomorphism rings of arbitrary direct sums of copies of a right hereditary ring. 

The general theory of (dual) Rickart objects in abelian categories may be efficiently applied both in the study of regular objects and in the study of (dual) Baer objects in abelian categories, as underlined in \cite{CK,CO}. Moreover, the framework of abelian categories allows an extensive use of the duality principle in order to automatically obtain dual results, and to investigate in a unified way concepts 
that have been independently studied in the literature. In subsidiary, one also has consequences in particular abelian categories, other than module categories.

Going back to module theory, the study of extending modules (also called CS-modules) 
and lifting modules has been a fruitful field of research for the last decades,
due to their important applications to ring and module theory.  
Examples of extending modules include uniform modules and injective modules,
while examples of lifting modules include hollow modules and projective modules over perfect rings.
The reader is referred to the monographs \cite{CLVW,DHSW} for further information on extending and lifting modules.
Recently, Abyzov, Nhan and Quynh have introduced and studied the concepts of CS-Rickart and dual CS-Rickart modules 
\cite{AN1,ANQ}, while Tribak has considered dual CS-Rickart modules over Dedekind domains \cite{Tribak}.  
Note that CS-Rickart modules represent a common generalization of Rickart modules and extending modules, 
while dual CS-Rickart modules represent a common generalization of  
dual Rickart modules and lifting modules. 

Motivated by all the above, we introduce and study (dual) CS-Rickart objects in abelian categories. 
To this end, we shall take advantage of some techniques used for (dual) Rickart objects, and we shall develop some new ones, inspired by the theory of extending and lifting modules. We depict in the following diagram the main concepts of the paper and the way in which they relate to each other. 

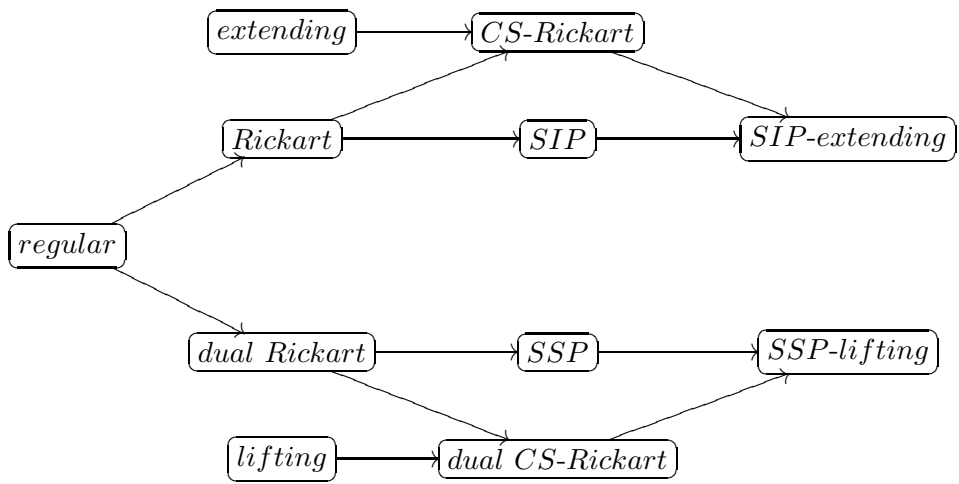
\begin{small}
$$\SelectTips{cm}{} 
\xymatrix{
 & *+[F-:<3pt>]{extending}  \ar[r] & *+[F-:<3pt>]{CS\textrm{-}Rickart} \ar[dr] &  \\
 & *+[F-:<3pt>]{Rickart} \ar[ur] \ar[r] & *+[F-:<3pt>]{SIP} \ar[r] & *+[F-:<3pt>]{SIP\textrm{-}extending} \\ 
*+[F-:<3pt>]{regular} \ar[ur] \ar[dr] & &  & \\
 & *+[F-:<3pt>]{dual \ Rickart} \ar[dr] \ar[r] & *+[F-:<3pt>]{SSP} \ar[r] & *+[F-:<3pt>]{SSP\textrm{-}lifting} \\
 & *+[F-:<3pt>]{lifting}  \ar[r] & *+[F-:<3pt>]{dual \ CS\textrm{-}Rickart} \ar[ur] &  \\
}$$
\end{small}

\vspace{2mm}

In what follows we briefly present the main results of the paper, by referring only to the case of CS-Rickart objects. 
The proofs of our results will also be presented only for CS-Rickart objects, 
the dual ones following by the duality principle in abelian categories.
In Section 2 we introduce relative CS-Rickart objects and self-CS-Rickart objects, and illustrate and delimit our concepts. 
Also, we discuss their relationship with Rickart objects in terms of some nonsingularity condition. In this direction, it turns out that self-Rickart objects are precisely the $\mathcal{K}$-nonsingular self-CS-Rickart objects.
Section 3 shows that the class of CS-Rickart objects is well behaved with respect to direct summands. It is known that self-Rickart objects have the summand intersection property (SIP). We prove that every self-CS-Rickart object has SIP-extending, a suitable property on direct summands which generalizes SIP. 
In Section 4 we show that if $M$ and $N_1,\dots,N_n$ are objects of an abelian category $\mathcal{A}$, 
then $\bigoplus_{i=1}^n N_i$ is $M$-CS-Rickart if and only if $N_i$ is $M$-CS-Rickart for every $i\in \{1,\dots,n\}$.
In general, coproducts of two self-CS-Rickart objects need not be self-CS-Rickart, but we have the following result. 
If $M=\bigoplus_{i\in I}M_i$ is a direct sum decomposition in an abelian category $\mathcal{A}$ 
such that $\Hom_{\mathcal{A}}(M_i,M_j)=0$ for every $i,j\in I$ with $i\neq j$, 
then $M$ is self-CS-Rickart if and only if $M_i$ is self-CS-Rickart for each $i\in I$.
Section 5 deals with classes all of whose objects are self-CS-Rickart. For instance, 
for an abelian category $\A$ with enough injectives, and a class $\mathcal{C}$ of objects of $\A$
which is closed under binary direct sums and contains all injective objects of $\A$, we prove that every object of $\C$ is extending
if and only if every object of $\C$ is self-CS-Rickart. We also deduce characterizations of 
right perfect and weakly (semi)hereditary rings in terms of self-CS-Rickart or dual self-CS-Rickart properties. 

The behaviour of CS-Rickart properties via functors between abelian categories is studied in \cite{CR2}. 
One may also consider a specialization of the concept of (dual) CS-Rickart object in an abelian category,  when considering fully invariant sections (fully coinvariant retractions) in the definition of a (dual) CS-Rickart object (see \cite{CR3}).

\section{(Dual) relative CS-Rickart objects}

Let $\mathcal{A}$ be an abelian category. For every morphism $f:M\to N$ in $\mathcal{A}$ we denote by
${\rm ker}(f):{\rm Ker}(f)\to M$, ${\rm coker}(f):N\to {\rm Coker}(f)$, 
${\rm coim}(f):M\to {\rm Coim}(f)$ and ${\rm im}(f):{\rm Im}(f)\to N$ the kernel, the cokernel, 
the coimage and the image of $f$ respectively. Note that ${\rm Coim}(f)\cong {\rm Im}(f)$, because $\A$ is abelian. 
For a short exact sequence $0\to A\to B\to C\to 0$ in $\A$, we also write $C=B/A$.
A morphism $f:A\to B$ is called a \emph{section} if there is a morphism
$f':B\to A$ such that $f'f=1_A$, and a \emph{retraction} if there is a morphism $f':B\to A$ such that $ff'=1_B$.
We recall a series of other necessary concepts. 

\begin{defn}[{\cite[Definition~2.2]{CK}}] \rm Let $M$ and $N$ be objects of an abelian category $\mathcal{A}$. Then $N$ is called:
\begin{enumerate}
\item \emph{$M$-Rickart} if the kernel of every morphism $f:M\to N$ is a section, or equivalently, 
the coimage of every morphism $f:M\to N$ is a retraction.
\item \emph{dual $M$-Rickart} if the cokernel of every morphism $f:M\to N$ is a retraction, or equivalently, 
the image of every morphism $f:M\to N$ is a section.
\item \emph{self-Rickart} if $N$ is $N$-Rickart.
\item \emph{dual self-Rickart} if $N$ is dual $N$-Rickart.
\end {enumerate}
\end{defn}

Since we work in abelian categories as in \cite{CK,DNTD}, we keep their terminology, in contrast to that of \cite{LRR10,LRR11}, which defines a module $M$ to be $N$-Rickart if the
kernel of every homomorphism $f:M\to N$ is a direct summand of $M$, and dual $N$-Rickart if the image of every homomorphism $f:M\to N$ is a direct summand of $N$. Let us also bear in mind that (dual) self-Rickart modules are simply called (dual) Rickart in \cite{LRR10,LRR11}. 

Essential monomorphisms and superfluous epimorphisms are well-known notions in categories, 
and they are defined as follows.

\begin{defn}\rm Let $\mathcal{A}$ be an abelian category. 
\begin{enumerate}
\item \rm A monomorphism $f:M \to N$ in $\mathcal{A}$ is called \textit{essential} 
if for every morphism $h:N \to P$ in $\mathcal{A}$ such that $hf$ is a monomorphism, $h$ is a monomorphism.     
Equivalently, a monomorphism $f:M\to N$ in $\mathcal{A}$ is essential if and only if 
${\rm Im}(f)$ is an essential subobject of $N$, in the sense that for any subobject $X$ of $M$, ${\rm Im}(f)\cap X=0$ implies $X=0$. 

An object is called \emph{uniform} if every non-zero subobject is essential.
\item \rm An epimorphism $f:M \to N$ in $\mathcal{A}$ is called \textit{superfluous} 
if for every morphism $h:P \to M$ in $\mathcal{A}$ such that $fh$ is an epimorphism, $h$ is an epimorphism. 
Equivalently, an epimorphism $f:M\to N$ in $\mathcal{A}$ is superfluous if and only if 
${\rm Ker}(f)$ is a superfluous subobject of $M$, in the sense that for any subobject $X$ of $M$, ${\rm Ker}(f)+X=M$ implies $X=M$. 

An object is called \emph{hollow} if every proper subobject is superfluous.
\end{enumerate}             
\end{defn}

The concepts of extending modules (also called CS-modules) and lifting modules \cite{CLVW,DHSW} 
are naturally generalized to abelian categories.

\begin{defn} \rm An object $M$ of an abelian category $\mathcal{A}$ is called:
\begin{enumerate}
\item \emph{extending} if every subobject of $M$ is essential in a direct summand of $M$.
\item \emph{lifting} if every subobject $L$ of $M$ lies above a direct summand of $M$, 
in the sense that $L$ contains a direct summand $K$ of $M$ such that $L/K$ is superfluous in $M/K$. 
\end{enumerate}
\end{defn}

Now we introduce the main concepts of the paper, which generalize both (dual) relative Rickart objects 
and extending (lifting) objects in abelian categories.

\begin{defn} \rm Let $M$ and $N$ be objects of an abelian category $\mathcal{A}$. Then $N$ is called:
\begin{enumerate}
\item \rm {\textit{$M$-CS-Rickart}} if for every morphism $f:M\to N$ there are an 
essential monomorphism $e:{\rm Ker}(f) \to L$ and a section $s:L \to M$ in $\mathcal{A}$ such that ${\rm ker}(f)=se$.
Equivalently, $N$ is $M$-CS-Rickart if and only if for every morphism $f:M\to N$, 
${\rm Ker}(f)$ is essential in a direct summand of $M$.
\item \rm{\textit{dual $M$-CS-Rickart}} if for every morphism $f:M\to N$ there are a    
retraction $r:N \to P$ and a superfluous epimorphism $t:P\to{\rm Coker}(f)$ in $\mathcal{A}$ such that ${\rm coker}(f)=tr$.
Equivalently, $N$ is dual $M$-CS-Rickart if and only if for every morphism $f:M\to N$, 
${\rm Im}(f)$ lies above a direct summand of $N$.
\item \rm {\textit{self-CS-Rickart}} if $N$ is $N$-CS-Rickart.
\item \rm  {\textit{dual self-CS-Rickart}} if $N$ is dual $N$-CS-Rickart.
\end{enumerate}
\end{defn}

\begin{rem} \rm (1) It is clear from the definitions that every (dual) self-Rickart object and 
every extending (lifting) object of an abelian category $\A$ is (dual) self-CS-Rickart. An object $M$ of $\A$ is extending if and only if every object $N$ of $\A$ is $M$-CS-Rickart, while an object $N$ of $\A$ is lifting if and only if $N$ is dual $M$-CS-Rickart for every object $M$ of $\A$. Also, an object $N$ of $\A$ is $M$-CS-Rickart for every uniform object $M$ of $\A$, while any hollow object $N$ of $\A$ is dual $M$-CS-Rickart for every object $M$ of $\A$. 

(2) It is well known that the objects of an abelian category need not have injective envelopes or projective covers. Nevertheless, certain (dual) relative CS-Rickart objects of abelian categories give rise to such envelopes and covers as follows. Let $M$ and $N$ be objects of an abelian category $\A$. 

If $M$ is injective and $N$ is $M$-CS-Rickart, then the kernel of any morphism $f:M\to N$ has an injective envelope (i.e., an essential monomorphism ${\rm Ker}(f)\to L$ with $L$ injective). Indeed, the existence of a section $s:L\to M$ and an essential monomorphism $e:{\rm Ker}(f) \to L$ such that ${\rm ker}(f)=se$ implies that $L$ is injective, and consequently $L$ is an injective envelope of ${\rm Ker}(f)$. In particular, the kernel of every endomorphism of an injective self-CS-Rickart object has an injective envelope. 

Dually, if $N$ is projective and dual $M$-CS-Rickart, then the cokernel of any homomorphism $f:M\to N$ has a projective cover (i.e., a superfluous epimorphism $P\to {\rm Coker}(f)$ with $P$ projective). In particular, the cokernel of every endomorphism of a projective dual self-CS-Rickart object has a projective cover.
\end{rem}

\begin{ex} \label{ex1} \rm We first give a list of examples related to (dual) self-CS-Rickart objects in the category 
of abelian groups ($\mathbb{Z}$-modules).

(i) $\mathbb{Z}_4=\mathbb{Z}/4\mathbb{Z}$ is both self-CS-Rickart and dual self-CS-Rickart, 
but it is neither self-Rickart, nor dual self-Rickart. 
Indeed, $\mathbb{Z}_4$ is clearly extending and lifting \cite[Corollary~22.4]{CLVW}, 
and thus $\mathbb{Z}_4$ is self-CS-Rickart and dual self-CS-Rickart.
Considering $f\in \End_\mathbb{Z}(\mathbb{Z}_4)$ defined by $f(x+4\mathbb{Z})=2x+4\mathbb{Z}$, 
note that ${\rm Ker}(f)={\rm Im}(f)=2\mathbb{Z}_4$ is not a direct summand of $\mathbb{Z}_4$. 
Hence $\mathbb{Z}_4$ is neither self-Rickart, nor dual self-Rickart.

(ii) $\mathbb{Z}$ is self-CS-Rickart, but not dual self-CS-Rickart. 
Indeed, $\mathbb{Z}$ is self-Rickart, and thus it is self-CS-Rickart. 
Considering $f\in \End_\mathbb{Z}(\mathbb{Z})$ defined by $f(x)=2x$, ${\rm Im}(f)=2\mathbb{Z}$ 
is not superfluous in the indecomposable abelian group $\mathbb{Z}$. Hence 
${\rm Im}(f)$ does not lie above a direct summand of $\mathbb{Z}$, and thus 
$\mathbb{Z}$ is not dual self-CS-Rickart. 
           
(iii) $\mathbb{Z}_4$ is both $\mathbb{Z}$-CS-Rickart and dual $\mathbb{Z}$-CS-Rickart, 
but neither $\mathbb{Z}$-Rickart nor dual $\mathbb{Z}$-Rickart.
Indeed, for every $f\in \Hom_\mathbb{Z}(\mathbb{Z}, \mathbb{Z}_4)$, 
${\rm Ker}(f)\in \{\mathbb{Z}, 2\mathbb{Z}, 4\mathbb{Z}\}$ is essential in $\mathbb{Z}$, because $\mathbb{Z}$ is uniform. 
Hence $\mathbb{Z}_4$ is $\mathbb{Z}$-CS-Rickart. On the other hand, for every $f\in \Hom_\mathbb{Z}(\mathbb{Z}, \mathbb{Z}_4)$, 
${\rm Im}(f)\in\{0, 2\mathbb{Z}_4, \mathbb{Z}_4\}$ and $\mathbb{Z}_4$ is hollow.  
Hence $\mathbb{Z}_4$ is dual $\mathbb{Z}$-CS-Rickart. Also, for $f\in \Hom_\mathbb{Z}(\mathbb{Z},\mathbb{Z}_4)$ 
defined by $f(x)=2x+4\mathbb{Z}$, ${\rm Ker}(f)=2\mathbb{Z}$ is not a direct summand of $\mathbb{Z}$ and   
${\rm Im}(f)=2\mathbb{Z}_4$ is not a direct sumand of $\mathbb{Z}_4$. Hence $\mathbb{Z}_4$ is neither $\mathbb{Z}$-Rickart,
nor dual $\mathbb{Z}$-Rickart.

(iv) $\mathbb{Z}\oplus \mathbb{Z}_p$ (for some prime $p$) is self-CS-Rickart, but not extending \cite[Example~1]{AN1}. Also, $\mathbb{Z}^{(\mathbb{R})}$ is self-Rickart, and thus self-CS-Rickart, but not extending \cite[Remark~2.28]{LRR10}.
                
(v) $\mathbb{Q}$ is dual self-CS-Rickart, but not lifting \cite[Example~2.14]{Tribak}.                
\end{ex}

\begin{ex} \label{ex2} \rm Consider the ring $R=\begin{pmatrix}\mathbb{Z}&\mathbb{Z}\\0&\mathbb{Z} \end{pmatrix}$. 
The direct summands of the right $R$-module $R$ are the following:
$$\begin{pmatrix} 0&0\\0&0 \end{pmatrix}, \begin{pmatrix}\mathbb{Z}&\mathbb{Z}\\0&\mathbb{Z} \end{pmatrix},
\begin{pmatrix}\mathbb{Z}&\mathbb{Z}\\0&0 \end{pmatrix}, \begin{pmatrix}0&n\\0&1 \end{pmatrix}\mathbb{Z} \quad (n\in \mathbb{Z}).$$
Denote $a=\begin{pmatrix}2&1\\0&0 \end{pmatrix}$ and let $t_a:R\to R$ be defined by $t_a(b)=ab$.  
Then ${\rm Ker}(t_a)=\begin{pmatrix}0&1\\0&-2 \end{pmatrix}\mathbb{Z}$. 
Since $R$ is the only direct summand of $R$ which contains ${\rm Ker}(t_a)$, and ${\rm Ker}(t_a)$ is not essential in $R$,
it follows that $R$ is not a self-CS-Rickart right $R$-module. 
Also, $\begin{pmatrix} 0&0\\0&0 \end{pmatrix}$ is the only direct summand of $R$ contained in 
${\rm Im}(t_a)=\begin{pmatrix}2\mathbb{Z}&\mathbb{Z}\\0&0 \end{pmatrix}$, 
and ${\rm Im}(t_a)$ is not superfluous in $R$. Hence ${\rm Im}(t_a)$ does not lie above a direct summand of $R$,
and thus $R$ is not a dual self-CS-Rickart right $R$-module.
\end{ex}

\begin{ex} \label{ex3} \rm Let $K$ be a field and consider the ring $R=\begin{pmatrix}K&K[X]\\0&K[X] \end{pmatrix}$. 
The direct summands of the right $R$-module $R$ are generated by the idempotents of ${\rm End}_R(R)\cong R$. 
The idempotents of $R$ are the following: 
$$\begin{pmatrix}0&0\\0&0 \end{pmatrix}, \begin{pmatrix}1&0\\0&1 \end{pmatrix}, 
\begin{pmatrix}1&f\\0&0 \end{pmatrix}, \begin{pmatrix}0&f\\0&1 \end{pmatrix} \quad (f\in K[X]).$$
Hence the direct summands of the right $R$-module $R$ are the following: 
$$\begin{pmatrix}0&0\\0&0 \end{pmatrix}, \begin{pmatrix}K&K[X]\\0&K[X] \end{pmatrix}, 
\begin{pmatrix}K&K[X]\\0&0 \end{pmatrix}, \begin{pmatrix}0&f\\0&1 \end{pmatrix}K[X] \quad (f\in K[X]).$$
Every endomorphism of the right $R$-module $R$ is of the form $t_a:R\to R$ given by $t_a(b)=ab$ 
for some $a=\begin{pmatrix}k&f\\0&g \end{pmatrix}\in R$. We distinguish the following cases:

Case 1. If $k\neq 0$ and $g\neq 0$, then ${\rm Ker}(t_a)=\begin{pmatrix}0&0\\0&0 \end{pmatrix}$. 

Case 2. If $k\neq 0$ and $g=0$, then ${\rm Ker}(t_a)=\begin{pmatrix}0&f\\0&1 \end{pmatrix}K[X]$ for some $f\in K[X]$. 

Case 3. If $k=0$ and at least one of $f$ and $g$ is non-zero, then ${\rm Ker}(t_a)=\begin{pmatrix}K&K[X]\\0&0 \end{pmatrix}$.

Case 4. If $k=0$, $f=0$ and $g=0$, then ${\rm Ker}(t_a)=R$.

In any case, ${\rm Ker}(t_a)$ is a direct summand of $R$. Hence $R$ is a self-Rickart right $R$-module,
and thus it is a self-CS-Rickart right $R$-module. 

Now let $a=\begin{pmatrix}0&0\\0&g \end{pmatrix}\in R$ with $g\notin K$. 
Then $\begin{pmatrix}0&0\\0&0 \end{pmatrix}$ is the only direct summand of $R$ contained in 
${\rm Im}(t_a)=\begin{pmatrix}0&0\\ 0&gK[X] \end{pmatrix}$, and ${\rm Im}(t_a)$ is not superfluous in $R$. 
Hence ${\rm Im}(t_a)$ does not lie above a direct summand of $R$,
and thus $R$ is not a dual self-CS-Rickart right $R$-module.
\end{ex}

In order to further relate (dual) relative Rickart and (dual) relative CS-Rickart objects we recall the following concepts.

\begin{defn}[{\cite[Definition~9.4]{CK}}] \rm Let $M$ and $N$ be objects of an abelian category $\mathcal{A}$. Then:
\begin{enumerate}
\item $N$ is called \emph{$M$-$\mathcal{K}$-nonsingular} if for any morphism $f:M\to N$ in $\mathcal{A}$, 
${\rm Ker}(f)$ essential in $M$ implies $f=0$. 
\item $M$ is called \emph{$N$-$\mathcal{T}$-nonsingular} if for any morphism $f:M\to N$ in $\mathcal{A}$, 
${\rm Im}(f)$ superfluous in $N$ implies $f=0$. 
\end{enumerate}
\end{defn}

The above notions are well behaved with respect to direct summands, as the following lemma shows.

\begin{lemm} \label{l:nonsing} Let $M$ and $N$ be objects of an abelian category $\mathcal{A}$, $M'$ a direct summand of $M$ 
and $N'$ a direct summand of $N$. 
\begin{enumerate}
\item If $N$ is $M$-$\mathcal{K}$-nonsingular, then $N'$ is $M'$-$\mathcal{K}$-nonsingular.
\item If $M$ is $N$-$\mathcal{T}$-nonsingular, then $M'$ is $N'$-$\mathcal{T}$-nonsingular
\end{enumerate}
\end{lemm}

\begin{proof} (1) Assume that $N$ is $M$-$\mathcal{K}$-nonsingular. 
Write $M=M'\oplus M''$ and $N=N'\oplus N''$. Let $f:M'\to N'$ be a morphism such that 
${\rm Ker}(f)$ is essential in $M'$. Consider the morphism $g=f\oplus 0:M'\oplus M''\to N'\oplus N''$. 
Then ${\rm Ker}(g)={\rm Ker}(f)\oplus M''$ is essential in $M=M'\oplus M''$. 
Since $N$ is $M$-$\mathcal{K}$-nonsingular, we have $g=0$, which implies $f=0$. 
Hence $N'$ is $M'$-$\mathcal{K}$-nonsingular. 
\end{proof}

The following theorem generalizes \cite[Lemmas~6,7]{AN1}.

\begin{theo} \label{t:nonsing} Let $M$ and $N$ be objects of an abelian category $\mathcal{A}$. Then:
\begin{enumerate}
\item $N$ is $M$-CS-Rickart and $N$ is $M$-$\mathcal{K}$-nonsingular if and only if $N$ is $M$-Rickart.
\item $N$ is dual $M$-CS-Rickart and $M$ is $N$-$\mathcal{T}$-nonsingular if and only if $N$ is dual $M$-Rickart.
\end{enumerate}
\end{theo}

\begin{proof} (1) Assume that $N$ is $M$-CS-Rickart and $N$ is $M$-$\mathcal{K}$-nonsingular. 
Let $f:M\to N$ be a morphism in $\A$. Since $N$ is $M$-CS-Rickart, there exists a direct summand $L$ of $M$ 
such that ${\rm Ker}(f)$ is essential in $L$. Let $i:L\to M$ be the canonical injection, 
and consider the morphism $fi:L\to N$. Since ${\rm Ker}(fi)={\rm Ker}(f)$ is essential in $L$ and 
$N$ is $L$-$\mathcal{K}$-nonsingular by Lemma \ref{l:nonsing}, we have $fi=0$. 
It follows that ${\rm Ker}(f)=L$. Hence ${\rm Ker}(f)$ is a direct summand of $M$, which shows that 
$N$ is $M$-Rickart.

Conversely, assume that $N$ is $M$-Rickart. Then $N$ is $M$-CS-Rickart. 
Now let $f:M\to N$ be a morphism in $\A$ such that ${\rm Ker}(f)$ is essential in $M$. 
But since $N$ is $M$-Rickart, ${\rm Ker}(f)$ is a direct summand of $M$. Hence ${\rm Ker}(f)=M$, and thus $f=0$.
Therefore, $N$ is $M$-$\mathcal{K}$-nonsingular. 
\end{proof}

Following \cite[8.1]{CLVW}, \cite[Section~4]{DHSW} and \cite{TV}, recall that a right $R$-module $M$ (for some unitary ring $R$) is called \emph{singular} if $M$ is isomorphic to $L/K$ for some module $L$ and essential submodule $K$ of $L$, and \emph{small} if $M$ is superfluous in some module $L$. Let $\mathcal{U}$ and $\mathcal{V}$ be the classes of singular right $R$-modules and small right $R$-modules respectively. For any right $R$-module $M$ let us denote by $Z(M)$ the trace of $\mathcal{U}$ in $M$ and by $\overline{Z}(M)$ the reject of $\mathcal{V}$ in $M$, that is, 
\begin{align*} Z(M)&={\rm Tr}(\mathcal{U},M)=\sum\{{\rm Im}(f)\mid f\in \Hom(U,M) \textrm{ for some } U\in \mathcal{U}\}, \\
\overline{Z}(M)&={\rm Re}(M,\mathcal{V})=\bigcap\{{\rm Ker}(f)\mid f\in \Hom(M,V) \textrm{ for some } V\in \mathcal{V}\}.
\end{align*}
Then $M$ is called \emph{non-singular} if $Z(M)=0$ and \emph{non-cosingular} if $\overline{Z}(M)=M$. Now it is clear that every non-singular right $R$-module $M$ is $M$-$\mathcal{K}$-nonsingular, while every non-cosingular right $R$-module $M$ is $M$-$\mathcal{T}$-nonsingular. Then Theorem \ref{t:nonsing} yields the following corollary. 

\begin{coll} \label{c:Z} Every non-singular self-CS-Rickart right $R$-module is self-Rickart and every non-cosingular dual self-CS-Rickart right $R$-module is dual self-Rickart.
\end{coll}

\section{Direct summands of (dual) relative CS-Rickart objects}

As in the case of (dual) relative Rickart objects and extending (lifting) objects, we see that
(dual) relative CS-Rickart objects are well behaved with respect to direct summands.

\begin{theo} \label{t:sdr1} Let $r:M\to M'$ be an epimorphism and $s:N'\to N$ a monomorphism in an abelian category $\mathcal{A}$. 
\begin{enumerate} 
\item If $r$ is a retraction and $N$ is $M$-CS-Rickart, then $N'$ is $M'$-CS-Rickart.
\item If $s$ is a section and $N$ is dual $M$-CS-Rickart, then $N'$ is dual $M'$-CS-Rickart.
\end{enumerate}
\end{theo}

\begin{proof}
(1) Let $f:M'\to N'$ be a morphism in $\mathcal{A}$. Consider the morphism $sfr: M\to N$. Let $L={\rm Ker}(fr)={\rm Ker}(sfr)$
and $l={\rm ker}(fr):L\to M$. Since $N$ is $M$-CS-Rickart, there exist an essential monomorphism 
$e:L\to U$ and a section $u:U\to M$ such that $l=ue$. Denote $X={\rm Ker}(r)$ and $K={\rm Ker}(f)$. 

Now we have the
following induced commutative diagram:
$$\SelectTips{cm}{}
\xymatrix{
 & 0 \ar[d] & 0 \ar[d] & 0\ar[d] & \\
 & X \ar@{=}[r] \ar[d]_i & X \ar@{=}[r] \ar[d]^{j} & X \ar[d]^{uj} \\ 
0 \ar[r] & L \ar[r]^-{e} \ar@{-->}[d]_p & U \ar[r]^-{u} \ar@{-->}[d]^-{q} & M \ar[r]^-{fr} \ar[d]^r & N' \ar@{=}[d] \ar[r]^s & N \ar@{=}[d] \\
0 \ar[r] & K \ar@{-->}[r]_-{k} \ar[d] & V \ar@{-->}[r]_-{v} \ar[d] & M' \ar[r]_f \ar[d] & N' \ar[r]_s & N \\
 & 0 & 0 & 0 &
}$$
with exact columns. Since the square $UMM'V$ is a pushout by \cite[Lemma~2.9]{CK}, 
it follows that $v$ is a section. Since the square $UMM'V$ and the rectangle $LMM'K$ are pullbacks by \cite[Lemma~2.9]{CK},
so is the square $LUVK$. It follows that $p$ and $q$ are retractions, and thus the first two vertical short exact sequences split. 
Then there exist $p':K\to L$ such that $pp'=1_K$ and $q':V\to U$ such that $qq'=1_V$. 
Then we have the following induced commutative diagram 
$$\SelectTips{cm}{}
\xymatrix{
0 \ar[r] & K \ar[r]^{p'} \ar[d]_k & L \ar[r] \ar[d]^e & X \ar[r] \ar@{=}[d] & 0 \\
0 \ar[r] & V \ar[r]_{q'} & U \ar[r] & X \ar[r] & 0
}
$$
with exact rows. Since the square $KLUV$ is a pullback by \cite[Lemma~2.9]{CK} and $e$ is an essential monomorphism,
it follows that $k$ is also an essential monomorphism by \cite[Chapter~5, Lemma~7.1]{St}.  
Hence ${\rm ker}(f)=vk$, where $k:K\to V$ is an essential monomorphism and $v:V\to M$ is a section. 
This shows that $N'$ is $M'$-CS-Rickart.
\end{proof}

The following consequence of Theorem \ref{t:sdr1} generalizes \cite[Lemma~1]{AN1} from the category of modules.

\begin{coll} \label{c:sdr5} Let $M$ and $N$ be objects of an abelian category $\mathcal{A}$, 
$M'$ a direct summand of $M$ and $N'$ a direct summand of $N$. 
\begin{enumerate} 
\item If $N$ is $M$-CS-Rickart, then $N'$ is $M'$-CS-Rickart.
\item If $N$ is dual $M$-CS-Rickart, then $N'$ is dual $M'$-CS-Rickart.
\end{enumerate}
\end{coll}

Recall that an object $M$ of an abelian category has the \emph{(strong) summand intersection property}, for short SIP (SSIP), 
if the intersection of any two (family of) direct summands of $M$ is a direct summand of $M$. The dual property is called the \emph{(strong) summand sum property}, for short SSP (SSSP). It is known that self-Rickart objects have SIP and dual self-Rickart objects have SSP \cite[Corollary~3.10]{CK}. In case of (dual) self-CS-Rickart objects, we need the following generalizations of SIP (SSIP) and SSP (SSSP) properties, inspired by the corresponding module-theoretic notions \cite{ANQ,Karab}. 

\begin{defn} \rm An object $M$ of an abelian category $\mathcal{A}$ has: 
\begin{enumerate}
\item \emph{SSIP-extending} (\emph{SIP-extending}) if for any family of (two) subobjects of $M$ 
which are essential in direct summands of $M$, their intersection is essential in a direct summand of $M$.
\item \emph{SSSP-lifting} (\emph{SSP-lifting}) if for any family of (two) subobjects of $M$ 
which lie above direct summands of $M$, their sum lies above a direct summand of $M$.
\end{enumerate}
\end{defn}

The following lemma, which generalizes \cite[Lemma~3.6]{ANQ}, will be implicitly used. 

\begin{lemm} \label{l:SIPSSP} Let $M$ be an object of an abelian category $\A$. Then:
\begin{enumerate}
\item $M$ has SIP-extending if and only if the intersection of any two direct summands of $M$ 
is essential in a direct summand of $M$. 
\item $M$ has SSP-lifting if and only if the sum of any two direct summands of $M$ lies above a direct summand of $M$.
\end{enumerate}
\end{lemm}

\begin{proof} (1) Assume that $M$ has SIP-extending. Let $M_1$ and $M_2$ be direct summands of $M$. 
Since each of them is essential in itself, it follows that $M_1\cap M_2$ is essential in a direct summand of $M$. 

Conversely, assume that the intersection of any two direct summands of $M$ 
is essential in a direct summand of $M$. Let $M_1$ and $M_2$ be subobjects of $M$ 
which are essential in direct summands $N_1$ and $N_2$ of $M$ respectively. Then $N_1\cap N_2$ is essential 
in a direct summand $N$ of $M$. Hence $M_1\cap M_2$ is essential in $N_1\cap N_2$, 
whence it follows that $M_1\cap M_2$ is essential in $N$.
\end{proof}

We point out that the version for SSIP-extending (SSSP-lifting) objects of Lemma \ref{l:SIPSSP} does not hold, 
because an arbitrary intersection (sum) of essential (superfluous) subobjects need not be an essential (superfluous) subobject.

\begin{prop} \label{p:sdr3} Let $M$ and $N$ be objects of an abelian category $\mathcal{A}$. 
\begin{enumerate}
\item Assume that every direct summand of $M$ is isomorphic to a subobject of $N$, and $N$ is $M$-CS-Rickart. 
Then $M$ has SIP-extending.
\item Assume that every direct summand of $N$ is isomorphic to a factor object of $M$, 
and $N$ is dual $M$-CS-Rickart. Then $N$ has SSP-lifting.
\end{enumerate}
\end{prop}

\begin{proof} (1) Let $M_1$ and $M_2$ be direct summands of $M$. 
Then $M_1\cong N_1$ and $M_2\cong N_2$ for some subobjects $N_1$ and $N_2$ of $N$.
Since $N$ is $M$-CS-Rickart, $N_1+N_2$ is $M$-CS-Rickart by Theorem~\ref{t:sdr1}. The canonical short exact sequence $$0\to
M_2\to M_1+M_2\to (M_1+M_2)/M_2\to 0$$ splits, because $M_2$ is a direct summand of $M$. Then there exists a
monomorphism $(M_1+M_2)/M_2\to M_1+M_2$, and thus a monomorphism $(N_1+N_2)/N_2\to N_1+N_2$. Then $N_1/(N_1\cap N_2)\cong
(N_1+N_2)/N_2$ is $M$-CS-Rickart by Theorem~\ref{t:sdr1}. Since $M_1$ is a direct summand of $M$, it follows that 
$N_1/(N_1\cap N_2)$ is $M_1$-CS-Rickart again by Theorem~\ref{t:sdr1}. Consider the induced short exact sequence 
$$0\to M_1\cap M_2\to M_1\to N_1/(N_1\cap N_2)\to 0.$$ Since $M_1$ is a direct summand of $M$ and $N_1/(N_1\cap N_2)$ is $M_1$-CS-Rickart, it follows that $M_1\cap M_2$ is essential in a direct summand of $M$. Then $M$ has SIP-extending by Lemma \ref{l:SIPSSP}. 
\end{proof}

The following consequence generalizes \cite[Propositions~1,2]{AN1}.

\begin{coll} \label{c:sdr4} Let $\mathcal{A}$ be an abelian category. 
\begin{enumerate}
\item Every self-CS-Rickart object of $\mathcal{A}$ has SIP-extending.
\item Every dual self-CS-Rickart object of $\mathcal{A}$ has SSP-lifting.
\end{enumerate}
\end{coll}

The following lemma, which generalizes \cite[Proposition~3.7]{ANQ}, will be useful.

\begin{lemm} \label{l:AB} Let $A$ and $B$ be objects of an abelian category $\A$.  
\begin{enumerate}
\item If $A\oplus B$ has SIP-extending, then $B$ is $A$-CS-Rickart.
\item If $A\oplus B$ has SSP-lifting, then $B$ is dual $A$-CS-Rickart.
\end{enumerate} 
\end{lemm}

\begin{proof} (1) Let $f:A\to B$ be a morphism in $\A$. 
Since $\left[\begin{smallmatrix} 1&0 \end{smallmatrix}\right]\left[\begin{smallmatrix} 1 \\ 0 \end{smallmatrix}\right]=1_A$ 
and $\left[\begin{smallmatrix} 1&0 \end{smallmatrix}\right]\left[\begin{smallmatrix} 1 \\ f \end{smallmatrix}\right]=1_A$,
the morphisms $\left[\begin{smallmatrix} 1 \\ 0 \end{smallmatrix}\right]:A\to A\oplus B$ 
and $\left[\begin{smallmatrix} 1 \\ f \end{smallmatrix}\right]:A\to A\oplus B$ are sections,
and thus ${\rm Im}\left(\left[\begin{smallmatrix} 1 \\ 0 \end{smallmatrix}\right]\right)$ and 
${\rm Im}\left(\left[\begin{smallmatrix} 1 \\ f \end{smallmatrix}\right]\right)$ are direct summands of $A\oplus B$.
Since $A\oplus B$ has SIP-extending, 
${\rm Im}\left(\left[\begin{smallmatrix} 1 \\ 0 \end{smallmatrix}\right]\right)\cap 
{\rm Im}\left(\left[\begin{smallmatrix} 1 \\ f \end{smallmatrix}\right]\right)$ is essential in a direct summand of $A\oplus B$.
We have the following commutative diagram
$$\SelectTips{cm}{}
\xymatrix{
0 \ar[r] & {\rm Ker}(f) \ar[r] \ar[d] & A \ar[r]^f \ar[d]^{\left[\begin{smallmatrix} 1 \\ 
f \end{smallmatrix}\right]} & B \ar@{=}[d] &  \\
0 \ar[r] & A \ar[r]_-{\left[\begin{smallmatrix} 1 \\ 0 \end{smallmatrix}\right]} & 
A\oplus B \ar[r]_-{\left[\begin{smallmatrix} 0 & 1 \end{smallmatrix}\right]} & B & 
}
$$
whose left square is a pullback and all the morphisms of the left square are monomorphisms. 
It follows that $${\rm Ker}(f)={\rm Im}\left(\left[\begin{smallmatrix} 1 \\ 0 \end{smallmatrix}\right]\right)\cap 
{\rm Im}\left(\left[\begin{smallmatrix} 1 \\ f \end{smallmatrix}\right]\right).$$ 
Then ${\rm Ker}(f)$ is essential in a direct summand of $A\oplus B$, and thus 
it is essential in a direct summand of $A$. Hence $B$ is $A$-CS-Rickart.
\end{proof}

\begin{coll} Let $M$ be an object of an abelian category $\A$.  
\begin{enumerate}
\item If $M\oplus M$ has SIP-extending, then $M$ is self-CS-Rickart.
\item If $M\oplus M$ has SSP-lifting, then $M$ is dual self-CS-Rickart.
\end{enumerate} 
\end{coll}

\begin{ex} \label{e:SIP} \rm Consider the $\mathbb{Z}$-module $M=\mathbb{Z}_2\oplus \mathbb{Z}_{16}$. 
Its direct summands are $\{(0,0)\}$, $M$ and the $\mathbb{Z}$-submodules $A=\langle(1,1)\rangle$, $B=\langle(0,1)\rangle$, 
$C=\langle(1,0)\rangle$ and $D=\langle(1,8)\rangle$ generated by the specified elements of $M$. 
We have $$M=A\oplus C=A\oplus D=B\oplus C=B\oplus D.$$ The intersection of any two direct summands of $M$ may be 
$\{(0,0)\}$, $M$, $A$, $B$, $C$, $D$ or $A\cap B$. Since $A\cap B$ is clearly essential 
in the local indecomposable $\mathbb{Z}$-module $A$, it follows that 
$M$ has SIP-extending by Lemma \ref{l:SIPSSP}. Then $\mathbb{Z}_{16}$ is $\mathbb{Z}_2$-CS-Rickart, 
and $\mathbb{Z}_2$ is $\mathbb{Z}_{16}$-CS-Rickart by Lemma \ref{l:AB}.
Note that $M$ does not have SIP (summand intersection property), because $A\cap B$ is not a direct summand of $M$.

We claim that $M$ is not self-CS-Rickart. To this end, let $f:M\to M$ be the homomorphism defined by 
$$f(a+2\mathbb{Z},b+16\mathbb{Z})=(b+2\mathbb{Z},2b+16\mathbb{Z}).$$ 
Then ${\rm Ker}(f)=\mathbb{Z}_2\oplus (8\mathbb{Z}/16\mathbb{Z})$. 
The only direct summand of $M$ containing ${\rm Ker}(f)$ is $M$, 
and ${\rm Ker}(f)$ is not essential in $M$. Hence $M$ is not self-CS-Rickart.
Therefore, there exist modules which have SIP-extending, but are not self-CS-Rickart.  
\end{ex}

\section{(Co)products of (dual) relative CS-Rickart objects}

In this section we analyze the behaviour of (dual) relative CS-Rickart and, in particular, (dual) self-CS-Rickart objects with respect to (co)products. 

We begin with the study of finite (co)products of (dual) relative CS-Rickart objects.

\begin{theo} \label{t:sdr2} Let $\mathcal{A}$ be an abelian category. 
\begin{enumerate} \item Let $M$, $N_1$ and $N_2$ be objects of $\mathcal{A}$
such that $N_1$ and $N_2$ are $M$-CS-Rickart. Then $N_1\oplus N_2$ is $M$-CS-Rickart.  
\item Let $M_1$, $M_2$ and $N$ be objects of $\mathcal{A}$ such that $N$ is
dual $M_1$-CS-Rickart and dual $M_2$-CS-Rickart. Then $N$ is dual $M_1\oplus M_2$-CS-Rickart. 
\end{enumerate}
\end{theo}

\begin{proof} (1) Denote $N=N_1\oplus N_2$. Let $f:M\to N$ be a morphism in $\mathcal{A}$. 
We have the following induced commutative diagram
$$\SelectTips{cm}{}
\xymatrix{
 & 0 \ar[d] & 0 \ar[d] & \\
 & K_1 \ar@{=}[r] \ar[d]_{i_1} & K_1 \ar[d]^{i_2} & \\ 
0 \ar[r] & K_2 \ar[r]^-{k_2} \ar[d]_{f_1} & M \ar[r]^-{f_2} \ar[d]^f & N_2 \ar@{=}[d] \ar[r] & 0 \\
0 \ar[r] & N_1 \ar[r]_{j_1} & N \ar[r]_{p_2} & N_2 \ar[r] & 0 
}$$
with exact rows and columns, where $j_1:N_1\to N$ is the canonical injection and $p_2:N\to N_2$ is 
the canonical projection. 

Since $N_2$ is $M$-CS-Rickart, there exist an essential monomorphism 
$e_2:K_2\to U$ and a section $u:U\to M$ such that $k_2=ue_2$. If $p_1:N\to N_1$ is the canonical projection,
then we have $p_1fue_2=f_1$. Then we have the following induced commutative diagram:
$$\SelectTips{cm}{}
\xymatrix{
0 \ar[r] & K_1 \ar[r]^-{i_1} \ar[d]_{e_1} & K_2 \ar[r]^{f_1} \ar[d]^{e_2} & N_1 \ar@{=}[d] \\ 
0 \ar[r] & K \ar[r]_k & U \ar[r]_-{p_1fu} & N_1 \\
}$$
with exact rows. Since the square $K_1K_2UK$ is a pullback by \cite[Lemma~2.9]{CK} and 
$e_2:K_2\to U$ is an essential monomorphism, it follows that $e_1:K_1\to K$ is also an essential monomorphism 
by \cite[Chapter~5, Lemma~7.1]{St}. Since $u$ is a section, there exists a retraction $u':M\to U$. 
But $N_1$ is $M$-CS-Rickart, hence $N_1$ is also $U$-CS-Rickart by Theorem \ref{t:sdr1}. 
Then there exist an essential monomorphism $e:K\to V$ and a section $v:V\to U$ such that $k=ve$. 
It follows that $ee_1:K_1\to V$ is an essential monomorphism, $uv:V\to M$ is a section and 
$${\rm ker}(f)=i_2=k_2i_1=ue_2i_1=uke_1=uvee_1.$$ This shows that $N$ is $M$-CS-Rickart.
\end{proof}

Now we may immediately deduce our main theorem on finite (co)products involving (dual) relative CS-Rickart objects.

\begin{theo} \label{t:pr1} Let $\mathcal{A}$ be an abelian category.
\begin{enumerate} 
\item Let $M$ and $N_1,\dots,N_n$ be objects of $\mathcal{A}$. Then $\bigoplus_{i=1}^n N_i$ is $M$-CS-Rickart if and only
if $N_i$ is $M$-CS-Rickart for every $i\in \{1,\dots,n\}$.
\item Let $M_1,\dots,M_n$ and $N$ be objects of $\mathcal{A}$. Then $N$ is dual $\bigoplus_{i=1}^n M_i$-CS-Rickart if and
only if $N$ is dual $M_i$-CS-Rickart for every $i\in \{1,\dots,n\}$.
\end{enumerate}
\end{theo}

\begin{proof} This follows by Corollary~\ref{c:sdr5} and Theorem~\ref{t:sdr2}.
\end{proof}

We illustrate the above result with an application to module categories. Following again \cite{DHSW,TV} and using the notation preceding Corollary \ref{c:Z}, for any right $R$-module $M$, denote by $Z_2(M)$ and $\overline{Z}^2(M)$ the submodules of $M$ determined by the equalities $$Z_2(M)/Z(M)=Z(M/Z(M)),$$ $$\overline{Z}^2(M)=\overline{Z}(\overline{Z}(M)).$$

\begin{coll} Let $M$ be a right $R$-module. Then:
\begin{enumerate}
\item $Z_2(M)$ and $M/Z_2(M)$ are $M$-CS-Rickart if and only if $Z_2(M)$ is a direct summand of $M$ and $M$ is self-CS-Rickart.
\item $M$ is dual $\overline{Z}^2(M)$-CS-Rickart and dual $M/\overline{Z}^2(M)$-CS-Rickart if and only if $\overline{Z}^2(M)$ is a direct summand of $M$ and $M$ is dual self-CS-Rickart.
\end{enumerate}
\end{coll}

\begin{proof} (1) Assume that $Z_2(M)$ and $M/Z_2(M)$ are $M$-CS-Rickart. Consider the natural epimorphism $p:M\to M/Z_2(M)$. By hypothesis, $Z_2(M)$ is essential in a direct summand $L$ of $M$. Since $M/Z_2(M)$ is non-singular, $Z_2(M)$ is closed, i.e., it does not have any proper essential extension \cite[4.1]{DHSW}. Hence $Z_2(M)=L$ is a direct summand of $M$. Then $M\cong Z_2(M)\oplus M/Z_2(M)$ is $M$-CS-Rickart by Theorem \ref{t:sdr2}. The converse follows by Corollary \ref{c:sdr5}.

(2) Assume that  $M$ is dual $\overline{Z}^2(M)$-CS-Rickart and dual $M/\overline{Z}^2(M)$-CS-Rickart. Consider the inclusion monomorphism $j:\overline{Z}^2(M)\to M$. By hypothesis, $\overline{Z}^2(M)$ lies above a direct summand $L$ of $M$. Since $\overline{Z}^2(M)$ is non-cosingular, $\overline{Z}^2(M)$ is coclosed, i.e., it does not lie above any proper submodule \cite[Lemma~2.3]{TV}. Hence $\overline{Z}^2(M)=L$ is a direct summand of $M$ and $M\cong \overline{Z}^2(M)\oplus M/\overline{Z}^2(M)$. Then $M$ is dual $M$-CS-Rickart by Theorem \ref{t:sdr2}. The converse follows by Corollary \ref{c:sdr5}.
\end{proof}

Under some finiteness conditions we have the following result.

\begin{coll} \label{c:pr2} Let $\mathcal{A}$ be an abelian category.
\begin{enumerate} \item Assume that $\mathcal{A}$ has coproducts, let $M$ be a finitely generated object of
$\mathcal{A}$, and let $(N_i)_{i\in I}$ be a family of objects of $\mathcal{A}$. Then $\bigoplus_{i\in I}
N_i$ is $M$-CS-Rickart if and only if $N_i$ is $M$-CS-Rickart for every $i\in I$.

\item Assume that $\mathcal{A}$ has products, let $N$ be a finitely cogenerated object of $\mathcal{A}$, and let
$(M_i)_{i\in I}$ be a family of objects of $\mathcal{A}$. Then
$N$ is dual $\prod_{i\in I} M_i$-CS-Rickart if and only if $N$ is dual $M_i$-CS-Rickart for every $i\in I$.
\end{enumerate}
\end{coll}

\begin{proof} (1) The direct implication follows by Corollary~\ref{c:sdr5}. For the converse, let $f:M\to
\bigoplus_{i\in I} N_i$ be a morphism in $\mathcal{A}$. Since $M$ is finitely generated, we may write
$f=jf'$ for some morphism $f':M\to \bigoplus_{i\in F} N_i$ and inclusion morphism $j:\bigoplus_{i\in F} N_i\to
\bigoplus_{i\in I} N_i$, where $F$ is a finite subset of $I$. By Theorem~\ref{t:pr1}, $\bigoplus_{i\in F} N_i$ is
$M$-CS-Rickart. Then there exist an essential monomorphism 
$e:{\rm Ker}(f')\to U$ and a section $u:U\to M$ such that ${\rm ker}(f')=ue$. 
But ${\rm ker}(f)={\rm ker}(f')$. Hence $\bigoplus_{i\in I} N_i$ is $M$-CS-Rickart.
\end{proof}

The next result follows by Theorem~\ref{t:sdr1} and gives a necessary condition for 
an infinite (co)product of objects to be a (dual) self-CS-Rickart object.

\begin{prop} \label{p:pr3} Let $(M_i)_{i\in I}$ be a family of objects of an abelian category $\mathcal{A}$.
\begin{enumerate} 
\item If $\prod_{i\in I} M_i$ is a self-CS-Rickart object, then $M_i$ is $M_j$-CS-Rickart for every $i,j\in I$.
\item If $\bigoplus_{i\in I} M_i$ is a dual self-CS-Rickart object, then $M_i$ is dual $M_j$-CS-Rickart for every
$i,j\in I$.
\end{enumerate}
\end{prop}

\begin{ex} \rm Consider the $\mathbb{Z}$-module $M=\mathbb{Z}_2\oplus \mathbb{Z}_{16}$. 
We have seen in Example \ref{e:SIP} that $\mathbb{Z}_{16}$ is $\mathbb{Z}_2$-CS-Rickart, 
and $\mathbb{Z}_2$ is $\mathbb{Z}_{16}$-CS-Rickart. Also, both $\mathbb{Z}_2$ and $\mathbb{Z}_{16}$ 
are self-CS-Rickart, because they are extending. On the other hand, $M=\mathbb{Z}_2\oplus \mathbb{Z}_{16}$ is not 
self-CS-Rickart by Example \ref{e:SIP}. This shows that the converse of Proposition \ref{p:pr3} does not hold in general.
\end{ex}

We also have the following theorem on arbitrary (co)products of (dual) relative CS-Rickart objects.

\begin{theo} Let $\mathcal{A}$ be an abelian category.
\begin{enumerate}
\item Let $M$ be an object of $\mathcal{A}$ having SSIP-extending, and let $(N_i)_{i\in I}$ be a family of objects of
$\mathcal{A}$ having a product. Then $\prod_{i\in I} N_i$ is
$M$-CS-Rickart if and only if $N_i$ is $M$-CS-Rickart for every $i\in I$.
\item Let $(M_i)_{i\in I}$ be a family of objects of $\mathcal{A}$ having a coproduct, and let $N$ be an object of
$\mathcal{A}$ having SSSP-lifting. Then $N$ is dual $\bigoplus_{i\in I} M_i$-CS-Rickart if and only if 
$N$ is dual $M_i$-CS-Rickart for every $i\in I$.
\end{enumerate}
\end{theo}

\begin{proof} (1) The direct implication follows by Theorem \ref{t:sdr1}. 
Conversely, assume that $N_i$ is $M$-CS-Rickart for every $i\in I$. Let $f:M\to \prod_{i\in I} N_i$
be a morphism in $\mathcal{A}$. For every $i\in I$, denote by $p_i:\prod_{i\in I} N_i\to N_i$ the canonical projection
and $f_i=p_if:M\to N_i$. Since $N_i$ is $M$-CS-Rickart, ${\rm Ker}(f_i)$ is essential in a direct summand of $M$ 
for every $i\in I$. Since $M$ has SSIP-extending, it follows that 
${\rm Ker}(f)=\bigcap_{i\in I} {\rm Ker}(f_i)$ is essential in a direct summand of $M$. 
Hence $\prod_{i\in I} N_i$ is $M$-CS-Rickart.
\end{proof}

In general the coproduct of two (dual) self-CS-Rickart objects is not (dual) self-CS-Rickart, 
as we may see in the following example.

\begin{ex} \label{ec1} \rm (i) Consider the ring $R=\begin{pmatrix}\mathbb{Z}&\mathbb{Z}\\0&\mathbb{Z} \end{pmatrix}$
and the right $R$-modules $M_1=\begin{pmatrix}\mathbb{Z}&\mathbb{Z}\\0&0 \end{pmatrix}$ 
and $M_2=\begin{pmatrix}0&0\\0&\mathbb{Z} \end{pmatrix}$. Since $\End(M_1)\cong \mathbb{Z} \cong \End(M_2)$,
$M_1$ and $M_2$ are self-Rickart, hence $M_1$ and $M_2$ are self-CS-Rickart right $R$-modules. 
But we have seen in Example \ref{ex2} that $R=M_1\oplus M_2$ is not a self-CS-Rickart right $R$-module. 

(ii) \cite[Example~2.6]{Tribak} The $\mathbb{Z}$-modules $\mathbb{Z}_2$ and $\mathbb{Z}_{16}$ are dual self-CS-Rickart, but the $\mathbb{Z}$-module $\mathbb{Z}_2\oplus \mathbb{Z}_{16}$ is not dual self-CS-Rickart.
\end{ex}

Nevertheless, we have the following result.

\begin{theo} \label{t:pstr4} Let $\mathcal{A}$ be an abelian category, 
and let $M=\bigoplus_{i\in I}M_i$ be a direct sum decomposition in $\mathcal{A}$ 
such that $\Hom_{\mathcal{A}}(M_i,M_j)=0$ for every $i,j\in I$ with $i\neq j$. Then:
\begin{enumerate}
\item $M$ is self-CS-Rickart if and only if $M_i$ is self-CS-Rickart for each $i\in I$.
\item $M$ is  dual self-CS-Rickart if and only if $M_i$ is dual self-CS-Rickart for each $i\in I$.
\end{enumerate}
\end{theo}

\begin{proof} (1) Assume first that $M$ is self-CS-Rickart. 
Then $M_i$ is self-CS-Rickart for every $i\in I$ by Corollary \ref{c:sdr5}.

Conversely, assume that $M_i$ is self-CS-Rickart for every $i\in I$. 
Let $f:M\to M$ be a morphism in $\mathcal{A}$. Its associated matrix has zero entries except for 
the entries $(i,i)$ with $i\in I$, which are some morphisms $f_i:M_i\to M_i$. Then $f=\bigoplus_{i\in I}f_i$
and $K={\rm Ker}(f)=\bigoplus_{i\in I}{\rm Ker}(f_i)$. Denote $k={\rm ker}(f):K\to M$ and 
$k_i={\rm ker}(f_i):K_i\to M_i$ for each $i\in I$.
Now let $i\in I$. Since $M_i$ is self-CS-Rickart, there are an essential monomorphism 
$u_i:K_i\to U_i$ and a section $s_i:U_i\to M_i$ such that $k_i=s_iu_i$.
Then $u=\bigoplus_{i\in I}u_i:\bigoplus_{i\in I}K_i\to \bigoplus_{i\in I}U_i$ is an essential monomorphism 
and $s=\bigoplus_{i\in I}s_i:\bigoplus_{i\in I}U_i\to \bigoplus_{i\in I}M_i$ is a section. 
We also have $k=su$, which shows that $M$ is self-CS-Rickart.

(2) This follows in a similar way as (1) by using images instead of kernels.
\end{proof}

\section{Classes all of whose objects are (dual) self-CS-Rickart}

In this section we obtain several characterizations of classes all of whose objects are (dual) self-CS-Rickart, 
mainly in connection with injective, projective, extending and lifting objects.

\begin{theo} \label{t:extlif} Let $\A$ be an abelian category. 
\begin{enumerate}
\item Assume that $\A$ has enough injectives. Let $\mathcal{C}$ be a class of objects of $\A$
which is closed under binary direct sums and contains all injective objects of $\A$. Then the following are equivalent:
\begin{enumerate}[(i)]
\item Every object of $\C$ is extending.
\item Every object of $\C$ is self-CS-Rickart.
\item Every object of $\C$ has SIP-extending.
\end{enumerate}
\item Assume that $\A$ has enough projectives. Let $\mathcal{C}$ be a class of objects of $\A$
which is closed under binary direct sums and contains all projective objects of $\A$. Then the following are equivalent:
\begin{enumerate}[(i)]
\item Every object of $\C$ is lifting.
\item Every object of $\C$ is dual self-CS-Rickart.
\item Every object of $\C$ has SSP-lifting.
\end{enumerate}
\end{enumerate} 
\end{theo}

\begin{proof} (1) (i)$\Rightarrow$(ii) This is clear.

(ii)$\Rightarrow$(iii) This follows by Corollary \ref{c:sdr4}.

(iii)$\Rightarrow$(i) Assume that every object of $\C$ has SIP-extending. 
Let $M$ be an object of $\mathcal{C}$. Let $N$ be a subobject of $M$, 
and denote by $p:M\to M/N$ the cokernel of the inclusion morphism $i:N\to M$,
and by $j:M/N\to E$ the inclusion into an injective object $E$ of $\A$. 
Consider the morphism $f=jp:M\to E$. Since $M\oplus E\in \C$, it has SIP-extending. 
Then $E$ is $M$-CS-Rickart by Lemma \ref{l:AB}. It follows that $N={\rm Ker}(f)$ is essential in a direct summand of $M$.
Hence $M$ is extending.
\end{proof}

\begin{coll} \label{c:extlif} Let $\A$ be an abelian category. 
\begin{enumerate}
\item The following are equivalent:
\begin{enumerate}[(i)]
\item Every object of $\A$ has an injective envelope.
\item $\A$ has enough injectives and every injective object of $\A$ is extending.
\item $\A$ has enough injectives and every injective object of $\A$ is self-CS-Rickart.
\item $\A$ has enough injectives and every injective object of $\A$ has SIP-extending.
\end{enumerate}
\item The following are equivalent:
\begin{enumerate}[(i)]
\item Every object of $\A$ has a projective cover (i.e., $\A$ is perfect).
\item $\A$ has enough projectives and every projective object of $\A$ is lifting.
\item $\A$ has enough projectives and every projective object of $\A$ is dual self-CS-Rickart.
\item $\A$ has enough projectives and every projective object of $\A$ has SSP-lifting.
\end{enumerate}
\end{enumerate} 
\end{coll}

\begin{proof} (1) This follows by Theorem \ref{t:extlif} with $\C$ the class of injective objects of $\A$, 
and by the dual of \cite[Theorem~3.5]{CK18}.
\end{proof}

Note that every object of a Grothendieck category $\A$ has an injective envelope, so every injective object of $\A$ 
is extending, and consequently self-CS-Rickart, by Corollary \ref{c:extlif}. 

As consequences of Theorem \ref{t:extlif} and Corollary \ref{c:extlif} for module categories we obtain the following partially known results (see \cite[Lemmas~5, 11]{AN1}).

\begin{coll} The following are equivalent for a unitary ring $R$ with Jacobson radical $J(R)$:
\begin{enumerate}[(i)]
\item Every right $R$-module is extending.
\item Every right $R$-module is self-CS-Rickart.
\item Every right $R$-module has SIP-extending.
\item Every right $R$-module is lifting.
\item Every right $R$-module is dual self-CS-Rickart.
\item Every right $R$-module has SSP-lifting.
\item $R$ is a left and right artinian serial ring with $(J(R))^2=0$. 
\end{enumerate}
\end{coll}

\begin{proof} The category ${\rm Mod}(R)$ of unitary right $R$-modules has enough injectives and enough projectives. 
Then the conclusion follows by Theorem \ref{t:extlif} with $\C={\rm Mod}(R)$ and by \cite[29.10]{CLVW}.
\end{proof}

\begin{coll} \label{c:perfect} The following are equivalent for a unitary ring $R$:
\begin{enumerate}[(i)]
\item $R$ is right perfect.
\item Every projective right $R$-module is lifting.
\item Every projective right $R$-module is dual self-CS-Rickart.
\item Every projective right $R$-module has SSP-lifting.
\end{enumerate}
\end{coll}

\begin{proof} This follows by Corollary \ref{c:extlif}.
\end{proof}

\begin{ex} \rm Consider the ring $R=\begin{pmatrix} \mathbb{Q}&0\\\mathbb{R}&\mathbb{R} \end{pmatrix}$. 
Since $R$ is right perfect (e.g., see \cite[Corollary~2.6]{HV}), 
$R$ is a dual self-CS-Rickart right $R$-module by Corollary \ref{c:perfect}. 
\end{ex}

Recall that a coalgebra $C$ over a field is called \emph{right perfect} if every right $C$-comodule has a projective cover \cite{CS}. 

\begin{coll} The following are equivalent for a coalgebra $C$ over a field:
\begin{enumerate}[(i)]
\item $C$ is right perfect.
\item $C$ is right semiperfect and every projective right $C$-comodule is lifting.
\item $C$ is right semiperfect and every projective right $C$-comodule is dual self-CS-Rickart.
\item $C$ is right semiperfect and every projective right $C$-comodule has SSP-lifting.
\end{enumerate}
\end{coll}

\begin{proof} The category of right $C$-comodules is a Grothendieck category \cite{DNR}, hence it has enough injectives.
It has enough projectives if and only if $C$ is right semiperfect \cite[Theorem~3.2.3]{DNR}. Then use Corollary \ref{c:extlif}.
\end{proof}

As in the module-theoretic case (e.g., see \cite[8.1]{CLVW} and \cite[Section~4]{DHSW}), we may consider the following notions.  

\begin{defn} \rm An object $M$ of an abelian category $\A$ is called:
\begin{enumerate} 
\item \emph{singular} if $M\cong L/K$ for some object $L$ of $\A$ and essential subobject $K$ of $L$. 
\item \emph{small} if $M$ is a superfluous subobject of some object $L$ of $\A$. 
\end{enumerate}
\end{defn}

Following \cite{AN1}, we generalize the concept of weakly (semi)hereditary module to abelian categories.

\begin{defn} \rm An object $M$ of an abelian category $\A$ is called: 
\begin{enumerate} 
\item \emph{weakly (semi)hereditary} if every (finitely generated) subobject of $M$ is the direct sum of 
a singular subobject and a projective subobject.
\item \emph{weakly (semi)cohereditary} if every (finitely cogenerated) factor object of $M$ is the direct sum of 
a small subobject and an injective subobject.
\end{enumerate}
\end{defn}

\begin{theo} \label{t:whc} Let $\A$ be a Grothendieck category. 
\begin{enumerate}
\item Assume that $\A$ has a family of finitely generated projective generators. 
Then the following are equivalent:
\begin{enumerate}[(i)]
\item Every (finitely generated) projective object of $\A$ is weakly (semi)hereditary.
\item Every (finitely generated) projective object of $\A$ is self-CS-Rickart.
\item Every (finitely generated) projective object of $\A$ has SIP-extending.
\end{enumerate}
\item Assume that $\A$ is locally finitely generated. Then the following are equivalent:
\begin{enumerate}[(i)]
\item Every (finitely cogenerated) injective object of $\A$ is weakly (semi)cohereditary.
\item Every (finitely cogenerated) injective object of $\A$ is dual self-CS-Rickart.
\item Every (finitely cogenerated) injective object of $\A$ has SSP-lifting.
\end{enumerate}
\end{enumerate} 
\end{theo}

\begin{proof} (1) (i)$\Rightarrow$(ii) Assume that every (finitely generated) projective object of $\A$ is weakly (semi)hereditary.
Let $M$ be a (finitely generated) projective object of $\A$. 
Let $f:M\to M$ be a morphism in $\A$. If $M$ is finitely generated, then so is ${\rm Im}(f)$. 
It follows that ${\rm Im}(f)=Q\oplus Z$ for some singular object $Q$ and projective object $Z$. 
Then we have the following induced commutative diagram
$$\SelectTips{cm}{}
\xymatrix{
 & & 0 \ar[d] & 0 \ar[d] & \\
0 \ar[r] & {\rm Ker}(f) \ar[r] \ar@{=}[d] & L \ar[r] \ar[d] & Q \ar[d] \ar[r] & 0 \\ 
0 \ar[r] & {\rm Ker}(f) \ar[r] & M \ar[r] \ar[d] & {\rm Im}(f) \ar[r] \ar[d] & 0 \\
 & & Z \ar@{=}[r] \ar[d] & Z \ar[d] & \\
 & & 0 & 0 & 
}$$
with exact rows and columns. Since $Z$ is projective, the middle vertical short exact sequence splits. 
Since $Q$ is singular and $L$ is projective, it follows that ${\rm Ker}(f)$ is essential in $L$ 
by \cite[Proposition~4.5]{DHSW}, whose proof is valid in any Grothendieck category.
Hence $M$ is self-CS-Rickart.

(ii)$\Rightarrow$(iii) This follows by Corollary \ref{c:sdr4}.

(iii)$\Rightarrow$(i) Assume that every (finitely generated) projective object of $\A$ has SIP-extending.
Let $M$ be a (finitely generated) projective object of $\mathcal{A}$. Let $N$ be a (finitely generated) subobject of $M$,
and denote by $i:N\to M$ the inclusion morphism. 
Since $\A$ has a family of finitely generated projective generators, 
there exists an epimorphism $p:P\to N$ for some (finitely generated) projective object $P$.
Consider the morphism $f=ip:P\to M$. Then $P\oplus M$ is a (finitely generated) projective object, hence it has SIP-extending by hypothesis. Then ${\rm Ker}(f)$ is essential in a direct summand $L$ of $P$ 
by Lemma \ref{l:AB}. Then we have an induced commutative diagram with exact rows and columns 
as in the proof of (i)$\Rightarrow$(ii), with $M$ replaced by $P$. Since the right-upper square is a pushout and 
the middle vertical short exact sequence splits, it follows that $N={\rm Im}(f)\cong Q\oplus Z$. 
Here $Q\cong L/{\rm Ker}(f)$ is singular and $Z$ is projective, because so is $P$. 
Hence $M$ is weakly (semi)hereditary.

(2) This follows in a dual manner as (1). We only point out that, since $\A$ is locally finitely generated Grothendieck, 
it has an injective cogenerator, namely $\bigoplus_S E(S)$, 
where $S$ runs over the isomorphism types of simple objects of $\A$ and 
$E(S)$ is the injective envelope of $S$ \cite[Lemma~E.1.11]{Prest}.
\end{proof} 
 
The following corollary extends a part of \cite[Theorem~8]{AN1}. Also, compare it with \cite[Theorems~3.11, 3.12]{ANQ} for nonsingular right (semi)hereditary rings. Note that a right weakly hereditary ring 
is the same as a right $\Sigma$-extending ring (or right co-$H$-ring) \cite[Corollary~11.13]{DHSW}.  
 
\begin{coll} \label{c:whc} Let $R$ be a unitary ring.
\begin{enumerate}
\item The following are equivalent:
\begin{enumerate}[(i)]
\item $R$ is right weakly (semi)hereditary.
\item Every (finitely generated) projective right $R$-module is weakly (semi)hereditary.
\item Every (finitely generated) projective right $R$-module is self-CS-Rickart.
\item Every (finitely generated) projective right $R$-module has SIP-extending.
\end{enumerate}
\item The following are equivalent:
\begin{enumerate}[(i)]
\item Every (finitely cogenerated) injective right $R$-module is weakly (semi)cohereditary.
\item Every (finitely cogenerated) injective right $R$-module is dual self-CS-Rickart.
\item Every (finitely cogenerated) injective right $R$-module has SSP-lifting.
\end{enumerate}
\end{enumerate} 
\end{coll} 

\begin{ex} \rm Let $Q$ be a local QF-ring with Jacobson radical $J(Q)$,
and consider the ring $R=\begin{pmatrix} Q&Q\\J(Q)&Q \end{pmatrix}$. 
Since $R$ is right weakly hereditary \cite[Theorem~5.5]{Oshiro}, 
$R$ is a self-CS-Rickart right $R$-module by Corollary \ref{c:whc}. 
\end{ex}
 
Recall that a coalgebra $C$ over a field is called \emph{semiperfect}
if every finitely generated right $C$-comodule has a projective cover \cite{DNR}. 

\begin{coll} \label{c:whc-coalg} Let $C$ be a coalgebra over a field.
\begin{enumerate}
\item Assume that $C$ is left and right semiperfect. Then the following are equivalent:
\begin{enumerate}[(i)]
\item Every (finitely generated) projective right $C$-comodule is weakly (semi)hereditary.
\item Every (finitely generated) projective right $C$-comodule is self-CS-Rickart.
\item Every (finitely generated) projective right $C$-comodule has SIP-extending.
\end{enumerate}
\item The following are equivalent:
\begin{enumerate}[(i)]
\item Every (finitely cogenerated) injective right $C$-comodule is weakly (semi)cohereditary.
\item Every (finitely cogenerated) injective right $C$-comodule is dual self-CS-Rickart.
\item Every (finitely cogenerated) injective right $C$-comodule has SSP-lifting.
\end{enumerate}
\end{enumerate} 
\end{coll}  
 
\begin{proof} The locally finitely generated Grothendieck category of right $C$-comodules 
has a family of finitely generated projective generators if $C$ is left and right semiperfect \cite[Section~3.2]{DNR}. 
Then use Theorem \ref{t:whc}. 
\end{proof}

\end{document}